\documentclass[reqno]{amsart}
\usepackage{amsthm,amsmath,amssymb,amsxtra}
\usepackage{bm,mathrsfs,marvosym,eurosym,enumitem}

\usepackage{hyperref,nameref}

\usepackage[top=2.3cm, bottom=2.3cm, left=3.2cm, right=3.2cm]{geometry}
\usepackage{cases,multicol}

\allowdisplaybreaks

\newenvironment{subex}[1]{%
\nobreak\vspace{-4pt}\nobreak\begin{multicols}{#1}%
\begin{enumerate}[label=(\arabic*),leftmargin=*,font=\normalfont]\itemsep=1pt%
}{\end{enumerate}\end{multicols}\vspace{-8pt}}

\numberwithin{equation}{section}

\def\C{{\mathbb C}}
\def\N{{\mathbb N}}
\def\R{{\mathbb R}}
\def\Z{{\mathbb Z}}

\def\ee{\mathrm{e}}
\def\ii{\mathrm{i}}

\DeclareMathOperator\Real{Re}

\def\ds{\displaystyle}

\newcommand{\high}{{\vphantom{|^|}}}
\newcommand{\low}{{\vphantom{|_|}}}

\theoremstyle{plain}
\newtheorem{theorem}{Theorem}
\newtheorem{lemma}{Lemma}

\theoremstyle{definition}

\renewcommand{\thefootnote}{\fnsymbol{footnote}}

\begin{document}

\title[Analyticity of dissipative-dispersive systems
in higher dimensions]{Analyticity of the attractors \\ of dissipative-dispersive systems
in higher dimensions}

\author[C.\ Evripidou]{Charalampos Evripidou}
\address{Department of Mathematics and Statistics, University of Cyprus, P.O. Box 20537, 1678 Nicosia, Cyprus\\
Department of Mathematics and Statistics, La Trobe University, Melbourne, Victoria 3086, Australia}
\email{cevrip02@ucy.ac.cy, c.evripidou@latrobe.edu.au}

\author[Y.-S.\ Smyrlis]{Yiorgos-Sokratis Smyrlis$^\dag$}
\address{Department of Mathematics and Statistics, University of Cyprus,
P.O. Box 20537, 1678 Nicosia, Cyprus}
\email{smyrlis@ucy.ac.cy}

\footnotetext[2]{The work of Y.-S. Smyrlis was supported
by an internal grant of the University of Cyprus.}

\renewcommand{\thefootnote}{\oldstylenums{\arabic{footnote}}}

\date{\today}

\begin{abstract}
We investigate the analyticity of the attractors of a class
of Kuramoto--Siva\-shinsky type pseudo-differential equations
in higher dimensions, which are periodic in all spatial
variables and possess a universal attractor. This is done
by fine-tuning the techniques used in \cite{APS2013,IS2016},
which are based on an analytic extensibility criterion
involving the growth of $\nabla^n u$, as $n$ tends to
infinity (here $u$ is the solution). These techniques can
now be utilised in a variety of higher dimensional equations
possessing universal attractors, including Topper--Kawahara
equation, Frenkel--Indireshkumar equations and their
dispersively modified analogs. We prove that the solutions
are analytic whenever $\gamma$, the {\it order of dissipation}
of the pseudo-differential operator,  is higher than one.
We believe that this estimate is optimal, based on numerical evidence. 
\end{abstract}

\subjclass[2010]{Primary: 35B65; Secondary: 35Q35.}

\keywords{Kuramoto--Sivashinsky equation, dissipative-dispersive evolutionary
equations, analyticity of solutions of partial differential equations, global attractors.}

\maketitle

\section{Introduction}

%

\noindent We are concerned with analyticity properties of solutions 
of pseudo-differential equations of the form
\begin{equation}
u_t+uu_{x_1}+{\mathcal P}u\,=\,0,
\label{eq:1}
\end{equation}
where $u=u(x_1,\ldots,x_d,t)$
with periodic initial data. 
We study equations in which such periodic 
initial value problems possess a universal attractor.
The pseudo-differential operator $\mathcal P$ is defined by its symbol
in the Fourier space. If the period in every spatial variable is $L$ 
then we assume that
$$
\big(\widehat{{\mathcal P}w} \big)_{\bm{k}}\,=\,\lambda_{\bm{k}} 
\,\hat{w}_{\bm{k}}, \quad \bm{k}=(k_1,\ldots,k_d)\in\Z^d,
$$
whenever 
$\,w(x_1,\ldots,x_k)=w(\bm{x})=\sum_{\bm{k}\in\Z^d}\hat{w}_{\bm{k}}\,\ee^{\ii q\bm{k}\cdot\bm{x}}\,$ 
and $\,q=2\pi/L$. 
We further assume that the eigenvalues $\lambda_{\bm{k}}$ satisfy the condition
\begin{equation}
\Real \lambda_{\bm{k}} \, \ge \, c_1 |\bm{k}|^\gamma-\mu, \quad \text{\,for all \, $\bm{k}\in\Z^d$},
\label{eq:eig}
\end{equation}
for some positive constants $\,c_1$, $\gamma$ and $\mu>0$.
This $\gamma$ is a lower bound of the order of the dissipative part of $\mathcal P$.
Our target, in this work, is to obtain an optimum lower bound of $\gamma$, which guarantees
analyticity for the universal attractor of \eqref{eq:1}.

\smallskip
\noindent The most typical example of an equation of the the form \eqref{eq:1}, when $d=1$,
is the Kuramoto-Sivashinsky (KS) equation
$$
u_t+uu_x+u_{xx}+u_{xxxx}=0,
$$
usually accompanied by $L-$periodic initial data. 
It is derived from a variety of physical models including
free surface film flows, 
two-phase flows in cylindrical and plane geometries flame-front instabilities
and reaction diffusion combustion dynamics
(see \cite{APS2013,CPS95} and references therein.)


\smallskip
\noindent There is a variety of higher dimensional physical models described by 
equation of the form \eqref{eq:1}. One such equation, when $d=2$,
is the Topper-Kawahara equation \cite{ToKa1978}
\begin{equation*}
u_t+uu_x+\alpha_1 u_{xx}+\alpha_2\Delta u+\alpha_3 \Delta^2 u+\alpha_4 \Delta u_{x} \,=\, 0,
\label{eq:KaTo}
\end{equation*}
where $u=u(x,y,t)$, which describes the evolution of liquid interface. Here 
$x$ is in the direction of the flow while $y$ is the transverse coordinate and $\alpha_3>0$. 
The particular case of this equation, where $\alpha_1=\alpha_3=1$ and $\alpha_2=0$, 
has been derived by Frenkel and Indireshkumar \cite{FI1999}.
In the case of interfacial instability of rotating core-annular flow a model also retaining dispersive effects 
has been derived by Coward and Hall \cite{CoHa1993} and takes the form
\begin{equation}
u_t+uu_x+\alpha\Delta u+\Delta^2 u + \delta {\mathcal D}u \,=\, 0, \label{eq:2d-D}
\end{equation}
where 
the two-dimensional pseudo-differential operator ${\mathcal D}$ is best represented in terms 
of its symbol in Fourier space given by 
$\, \widehat{{\mathcal D}u}(\xi,\eta)={\mathcal N}(\xi,\eta)\hat{u}(\xi,\eta)$, where
\begin{align*}
{\mathcal N}(\xi,\eta) =\,\,& \frac{2\ii\eta^2 I_\eta \big(\xi I_\eta^2-2\eta I_{\eta +1}
I_\eta -\xi I_\eta I_{\eta +1}\big)}{2\xi I_{\eta +1}^2I_{\eta -1}-\xi I_{\eta }^2I_{\eta -1} 
-\xi I_{\eta }^2I_{\eta +1}+2(2+\eta )I_\eta I_{\eta +1}I_{\eta -1}} \\ &+
\frac{ \ii \xi^2 I_{\eta +1} \big(\xi I_\eta I_{\eta -1} -2(\eta -2)I_{\eta -1}I_{\eta +1} 
-\xi I_\eta I_{\eta +1}\big)}
{2\xi I_{\eta +1}^2 I_{\eta -1}-\xi I_\eta^2 I_{\eta -1}-\xi I_\eta^2 I_{\eta +1} 
+2(2+\eta )I_\eta I_{\eta +1}I_{\eta -1}}+\ii\xi \eta,
\end{align*}
where $\xi,\eta$ denote the wave numbers in the Fourier transforms in the $x$ and $y$ directions, respectively, and 
$\, I_\eta=I_\eta(\xi)\,$ denotes the modified Bessel function of the first kind of order $\eta$ with $\,\eta\in\Z$.



\medskip
\noindent Pinto in \cite{Pinto1999,Pinto2001} studied the periodic initial value problem of\begin{equation}
u_t + uu_x + u_{xx} + \Delta^2 u =0,
\label{eq:Pinto}
\end{equation}
and proved existence of solutions, the existence of a universal attractor, and analyticity of solutions.
Sharp numerical estimates for the size of the attractor of \eqref{eq:Pinto} 
are presented in \cite{AKPS}, suggesting that
$\,\limsup_{t\rightarrow\infty} \| u(\cdot,\cdot,t) \| =\
\mathcal{O}(L^2),$
where $L$ is the length of the common spatial period.
Demirkaya \cite{De2009} studied the following 
two dimensional variation of the KS equation
\begin{equation*}
u_t +\Delta u + \Delta^2u  + uu_x + uu_y - g(x) \,=\, 0,
\end{equation*}
with $2\pi-$periodic initial data in both variables,
and established the global well-posedness and the existence of 
a universal attractor.
Ioakim \& Smyrlis \cite{IS2016} established the analyticity of the 
universal attractor of \eqref{eq:1} with periodic initial data in two space dimensions,
provided that $\gamma>3$.

\medskip
\noindent In what follows, we establish the analyticity of the universal attractor of \eqref{eq:1} 
when $\gamma>1$, by fine-tuning the methods presented in \cite{IS2016}.
Numerical experiments presented in \cite{Sm} suggest that this
bound is optimal.
 
\section{A criterion which guarantees analytic extensibility}

If  $f : \R^d \rightarrow \C$ is real analytic and $L-$periodic in every one of its $d$
variables, then it extends
as a holomorphic function, in each of its $d$ complex variables, in a domain of the form
\begin{equation}
W_\beta \,=\, \big\{ (z_1,\ldots,z_d)\in\C^d\,:\,|\mathrm{Im}\,z_j|<\beta\,\,\,
\text{for $j=1,\ldots,d$}
\big\}\subset\C^d,
\label{Wb}
\end{equation}
for some $\,\beta>0$.
The maximum such $\,\beta \!\in\! (0,\infty]$ we call {\it band of analyticity} of $f$.
We shall obtain a lower bound of the band of analyticity  
from the rate of growth, as a function of $s$, of the seminorm
\begin{equation}
\|f\|_{s,\infty}\,=\, \sup_{\bm{k}\in\Z^d} |\bm{k}|^s|\hat f_{\bm{k}}|,\quad
\text{where 
$\,\,\ds\hat{f}_{\bm{k}}\,=\,\frac{1}{L}\int_{[0,L]^d} f(\bm{x})\,{\ee}^{-\ii q\bm{k}\cdot\bm{x}} d\bm{x},
\,\,\,q=\frac{2\pi}{L}\,$}
\label{eq:seminorm}
\end{equation}
when $s$ tends to infinity.



\begin{lemma}
\label{lemma:bound}
Let $f:\R^k \rightarrow \C$ be a $\mathrm{C}^\infty$-function which is $2\pi$-periodic
in every argument and for which there exist $M,a>0$ such that
$$
\|f\|_{s,\infty}\leq M(as)^s, \text{ for all } s>0.
$$
Then $f$ extends holomorphically in $W_\beta$, where $\beta$ satisfies
$\beta \geq \ds\frac{1^\high}{\ee a\low}$.
\end{lemma}

\proof See Appendix. \hfill $\Box$

\section{Analyticity for higher dimensional models}


The main result of this paper is the following theorem:

\begin{theorem}
Assume that equation \eqref{eq:1}, with $2\pi-$periodic initial data 
in all its $d$ variables, possesses 
a universal attractor ${\mathcal V}$  and also that
the pseudo-differential operator $\mathcal P$ satisfies the condition
\eqref{eq:eig}, for some positive constants $\mu,c_1$ and some $\gamma>1$.
Then, there exists a $\beta>0$, such that every $w \in {\mathcal V}$ 
extends to a holomorphic function in $W_\beta$.
\end{theorem}

\begin{proof}
For simplicity, we do the proof in the two variable case ($d = 2$).
The extension to higher dimension follows the same steps. We denote by  $x$
and $y$ the  spatial
variables and by $(k,\ell)$ the elements of $\Z^2$.
Equation \eqref{eq:1} is transformed into the following  infinite dimensional
dynamical system
\begin{equation}
\frac{d}{dt}\hat{u}_{k,\ell} = -\lambda_{k,\ell} \hat{u}_{k,\ell} -\ii k\hat{\varphi}_{k,\ell},
\quad k,\ell\in\Z
\label{IDDS}
\end{equation}
where 
$\,u(x,y,t)= \sum_{k,\ell\in\Z}\hat{u}_{k,\ell}(t)\,\ee^{\ii(kx+\ell y)},\,$
and for $k,\ell\geq 0$, we have
\begin{gather*}
\hat{\varphi}_{k,\ell}(t) =\, \frac{1}{(2\pi)^2} \int_0^{2\pi} \int_0^{2\pi} \frac{1}{2}
u^2(x,y,t)\,\ee^{-\ii(kx+\ell y)}\,dxdy =\\
\frac{1}{2}\sum_{j=0}^k\sum_{m=0}^\ell\hat{u}_{j,m}\hat{u}_{k-j,\ell-m}+
\sum_{j=0}^k\sum_{m=1}^\infty\hat{u}_{j,-m}\hat{u}_{k-j,\ell+m}+
\sum_{j=1}^\infty\sum_{m=0}^\ell\hat{u}_{-j,m}\hat{u}_{k+j,\ell-m}+\\
\sum_{j=1}^\infty\sum_{m=1}^\infty\hat{u}_{-j,-m}\hat{u}_{k+j,\ell+m}+
\sum_{j=1}^\infty\sum_{m=1}^\infty\hat{u}_{-j,\ell+m}\hat{u}_{k+j,-m}.
\label{phi kl 0}
\end{gather*}
We get a similar formula for $\hat{\varphi}_{k,\ell}(t)$ when $k\geq0$ and $\ell<0$.
If $\mu_{k,\ell}=\mu+\lambda_{k,\ell}$ then \eqref{IDDS} becomes
$$
\frac{d}{dt}\hat{u}_{k,\ell}= -\mu_{k,\ell} \hat{u}_{k,\ell}+
\left(\mu\hat{u}_{k,\ell} -\ii k\hat{\varphi}_{k,\ell}\right),
\quad k,\ell\in\Z.
$$
We introduce now $\,w=\sum_{k,\ell\in\mathbb Z}\hat{w}_{k,\ell}(t)\,\ee^{\ii(kx+\ell y)},\,$
with
$\,\hat{w}_{k,\ell}(t)=(|k|+|\ell|)^3\hat{u}_{k,\ell}(t).\,$
Then
\begin{equation}
\label{eq:w_k,l_system}
\frac{d}{dt}\hat{w}_{k,\ell}= -\mu_{k,\ell} \hat{w}_{k,\ell}+
\hat{\vartheta}_{k,\ell}, \quad k,\ell\in\Z
\end{equation}
with
$\,\hat{\vartheta}_{k,\ell}=\mu\hat{w}_{k,\ell} -
\ii k(|k|+|\ell|)^3\hat{\varphi}_{k,\ell}$,
for all $k,\ell\in\Z.$
Clearly, \eqref{eq:w_k,l_system} implies that
\begin{equation*}
\hat{w}_{k,\ell}(t) \,=\, \ee^{-\mu_{k,\ell} t}\hat{w}_{k,\ell}(0)+\int_0^t
\ee^{-\mu_{k,\ell}(t-s)}\hat{\vartheta}_{k,\ell}(s)\,ds,
\end{equation*}
and consequently
\begin{equation}
\limsup_{t\rightarrow\infty} |\hat{w}_{k,\ell}(t)| \,\le\, \frac{1}{\Real \mu_{k,\ell}}
\limsup_{t\rightarrow\infty} |\hat{\vartheta}_{k,\ell}(t)|\,\le\,
\frac{1}{c_1(|k|+|\ell|)^\gamma}\limsup_{t\rightarrow\infty} |\hat{\vartheta}_{k,\ell}(t)|.
\label{phikl}
\end{equation}
We next define
$$
h(s) = \limsup_{t\rightarrow\infty} \sup_{k,\ell \in\Z} (|k|+|\ell|)^{s}
|\hat{w}_{k,\ell}(t)|,
\quad s\in\R.
$$
We set $\,M=\limsup_{t\rightarrow\infty}\| w(\cdot,\cdot,t) \|_0$. Then $h(0)\le M$.

\smallskip
\noindent We shall now establish the following {\it boot-strap} type argument: 

\smallskip
\noindent Assuming that for some $\sigma\ge0$
$$
h(s)\le M(as)^s\quad\text{for all $s\in [0,\sigma]$}
$$
we will show\footnote{Convention. The value of $(as)^s$ is equal to 1, when $s=0$.}
that $h(s)\le M(as)^s$ for all $s\in [0,\sigma+\gamma-1]$, and
consequently that $\,h(s)\le M(as)^s$ for all $s\ge 0$,
and finally the Theorem will be a corollary of Lemma \ref{lemma:bound}.

\smallskip
\noindent From our assumption we have, for every $k,\ell\in \Z
\text{ with } |k|+|\ell|\neq 0 \text{ and } s \!\in\! [0,\sigma]$, that
\begin{equation}
\label{w_est}
\limsup_{t\rightarrow\infty}|\hat{w}_{k,\ell}(t)|\le \frac{h(s)}{(|k|+|\ell|)^{s}}
\le \frac{M(as)^s}{(|k|+|\ell|)^{s}}.
\end{equation}
In order to prove that $h(s)\le M(as)^s$ for all $s\in [0,\sigma+\gamma-1]$
we will make a better estimate of $\limsup_{t\rightarrow\infty}|\hat{w}_{k,\ell}(t)|$
rather than \eqref{w_est}. It is evident that we may,
without loss of generality, assume $k,\ell\ge0$.
We will make use of the following useful inequality
\begin{equation}
\label{ineq}
\limsup_{t\rightarrow\infty}|\hat{w}_{j,m}(t)|\le
\frac{M(a\sigma\frac{r}{k+\ell})^{\sigma\frac{r}{k+\ell}}}{(j+m)^{\sigma\frac{r}{k+\ell}}}\le
M\left(\frac{a\sigma}{k+\ell}\right)^{\sigma\frac{r}{k+\ell}},
\end{equation}
for every $j,m,k,\ell,r\in \N$ with $0\le r \le \min\{j+m,k+\ell\}$.
We have
\begin{gather}
|\hat{\vartheta}_{k,\ell}(t)|\le
\mu|\hat{w}_{k,\ell}(t)| +
(k+\ell)\left(\frac{1}{2}\underset{0<j+m<k+\ell}
{\sum_{j=0}^k\sum_{m=0}^\ell}(k+\ell)^3\frac{|\hat{w}_{j,m}(t)|}{(j+m)^3}
\frac{|\hat{w}_{k-j,\ell-m}(t)|}{(k-j+\ell-m)^3}\right.\nonumber\\
+\sum_{j=0}^k\sum_{m=1}^\infty
(k+\ell)^3\frac{|\hat{w}_{j,m}(t)|}{(j+m)^3}
\frac{|\hat{w}_{k-j,\ell+m}(t)|}{(k-j+\ell+m)^3}+
\sum_{j=1}^\infty\sum_{m=0}^\ell
(k+\ell)^3\frac{|\hat{w}_{j,m}(t)|}{(j+m)^3}
\frac{|\hat{w}_{k+j,\ell-m}(t)|}{(k+j+\ell-m)^3}\nonumber\\
\label{ineq_theta}
+\left.\sum_{j=1}^\infty\sum_{m=1}^\infty
(k+\ell)^3\frac{|\hat{w}_{j,m}(t)|}{(j+m)^3}
\frac{|\hat{w}_{k+j,\ell+m}(t)|}{(k+j+\ell+m)^3}+
\sum_{j=1}^\infty\sum_{m=1}^\infty
(k+\ell)^3\frac{|\hat{w}_{j,\ell+m}(t)|}{(j+\ell+m)^3}
\frac{|\hat{w}_{k+j,m}(t)|}{(k+j+m)^3}\right).
\end{gather}
The sums appearing in the previous inequality are estimated as follows.
For the first sum, using \eqref{ineq} with $r=j+m$ for $\hat{w}_{j,m}(t)$
and with $r=k+\ell-j-m$ for the term $\hat{w}_{k-j,\ell-m}(t)$, we get
\begin{align*}
\limsup_{t\rightarrow\infty}&
\underset{0<j+m<k+\ell}
{\sum_{j=0}^k  \sum_{m=0}^\ell}
(k+\ell)^3\frac{|\hat{w}_{j,m}(t)|}{(j+m)^3}
\frac{|\hat{w}_{k-j,\ell-m}(t)|}{(k-j+\ell-m)^3}\\
\le M^2&\underset{0<j+m<k+\ell}{\sum_{j=0}^k\sum_{m=0}^\ell}
(k+\ell)^3\frac{\left(\frac{a\sigma}{k+\ell}\right)^{\sigma\frac{j+m}{k+\ell}}}{(j+m)^3}
\frac{\left(\frac{a\sigma}{k+\ell}\right)^{\sigma\frac{k-j+\ell-m}{k+\ell}}}
{(k-j+\ell-m)^3}\\
=M^2&\frac{(a\sigma)^\sigma}{(k+\ell)^\sigma}
\underset{0<j+m<k+\ell}
{\sum_{j=0}^k\sum_{m=0}^\ell}\frac{(k+\ell)^3}{(j+m)^3(k-j+\ell-m)^3}\le
M^2\frac{(a\sigma)^\sigma}{(k+\ell)^{\sigma}}R_1
\end{align*}
where
$$
R_1=\underset{k,\ell \in\N}{\sup}
{\underset{0<j+m<k+\ell}
{\sum_{j=0}^k\sum_{m=0}^\ell}\frac{(k+\ell)^3}{(j+m)^3(k-j+\ell-m)^3}}<\infty.
$$
In Lemma \ref{lem:2} (see Appendix) we prove that $R_1$ is indeed finite.
For the second sum, using again \eqref{ineq} with $r=j$ for $\hat{w}_{j,m}(t)$
and with $r=k+\ell-j$ for the term $\hat{w}_{k-j,\ell+m}(t)$, we get
\begin{align*}
\limsup_{t\rightarrow\infty}\sum_{j=0}^k\sum_{m=1}^\infty&
(k+\ell)^3\frac{|\hat{w}_{j,m}|}{(j+m)^3}
\frac{|\hat{w}_{k-j,\ell+m}|}{(k-j+\ell+m)^3}  \\ &\le
M^2\sum_{j=0}^k\sum_{m=1}^\infty
(k+\ell)^3\frac{\left(\frac{a\sigma}{k+\ell}\right)^{\sigma\frac{j}{k+\ell}}}{(j+m)^3}
\frac{\left(\frac{a\sigma}{k+\ell}\right)^{\sigma\frac{k+\ell-j}{k+\ell}}}
{(k-j+\ell+m)^3}
=M^2\frac{(a\sigma)^\sigma}{(k+\ell)^{\sigma}}R_2,
\end{align*}
where
$$
R_2=\underset{k,\ell \in\N}{\sup}{\sum_{j=0}^k\sum_{m=1}^\infty
\frac{(k+\ell)^3}{(j+m)^3(k-j+\ell+m)^3}}<\infty.
$$
In Lemma 2 we also prove that $R_2$ is finite.
Similarly we treat the rest sums and we get for the third one
\begin{align*}
\limsup_{t\rightarrow\infty}\sum_{j=1}^\infty & \sum_{m=0}^\ell
(k+\ell)^3\frac{|\hat{w}_{j,m}|}{(j+m)^3}
\frac{|\hat{w}_{k+j,\ell-m}|}{(k+j+\ell-m)^3} \\
&\le M^2\sum_{j=1}^\infty\sum_{m=0}^\ell
(k+\ell)^3\frac{\left(\frac{a\sigma}{k+\ell}\right)^{\sigma\frac{m}{k+\ell}}}{(j+m)^3}
\frac{\left(\frac{a\sigma}{k+\ell}\right)^{\sigma\frac{k+\ell-m}{k+\ell}}}
{(k+j+\ell-m)^3}
\le M^2\frac{(a\sigma)^\sigma}{(k+\ell)^{\sigma}}R_2
\end{align*}
and for the fourth
\begin{align*}
\limsup_{t\rightarrow\infty}
\sum_{j=1}^\infty\sum_{m=1}^\infty
&(k+\ell)^3\frac{|\hat{w}_{j,m}|}{(j+m)^3}
\frac{|\hat{w}_{k+j,\ell+m}|}{(k+j+\ell+m)^3} \\
&\le
M^2\frac{(a\sigma)^\sigma}{(k+\ell)^\sigma}\sum_{j=1}^\infty\sum_{m=1}^\infty
\frac{(k+\ell)^3}{(j+m)^3(k+j+\ell+m)^3}
\le M^2\frac{(a\sigma)^\sigma}{(k+\ell)^\sigma}R_3,
\end{align*}
with
$\,
R_3=\underset{k,\ell \in\N}{\sup}{}\sum_{j=1}^\infty\sum_{m=1}^\infty
\frac{(k+\ell)^3}{(j+m)^3(k+j+\ell+m)^3},\,
$
which is bounded (see Lemma 2). Finally, for the fifth sum
\begin{align*}
\limsup_{t\rightarrow\infty}
\sum_{j=1}^\infty\sum_{m=1}^\infty &
(k+\ell)^3\frac{|\hat{w}_{j,\ell+m}|}{(j+\ell+m)^3}
\frac{|\hat{w}_{k+j,m}|}{(k+j+m)^3} \\
&\le M^2\sum_{j=1}^\infty\sum_{m=1}^\infty
(k+\ell)^3\frac{\left(\frac{a\sigma}{k+\ell}\right)^{\sigma\frac{\ell}{k+\ell}}}
{(j+\ell+m)^3}
\frac{\left(\frac{a\sigma}{k+\ell}\right)^{\sigma\frac{k}{k+\ell}}}{(k+j+m)^3}
\le M^2\frac{(a\sigma)^\sigma}{(k+\ell)^{\sigma}}R_4,
\end{align*}
where, as by Lemma 2
$$
R_4=\underset{k,\ell \in\N}{\sup}{\sum_{j=1}^\infty\sum_{m=1}^\infty
\frac{(k+\ell)^3}{(j+\ell+m)^3(k+j+m)^3}}<\infty.
$$
Summing all these and using \eqref{phikl} and \eqref{ineq_theta} we get
\begin{align*}
\limsup_{t\rightarrow\infty} |\hat{w}_{k,\ell}(t)|  
&\le
\mu M \frac{(a\sigma)^\sigma}{c_1(k+\ell)^{\sigma+\gamma}}+
M^2(k+\ell)\frac{(a\sigma)^\sigma}{c_1(k+\ell)^{\sigma+\gamma}}
\left(\frac{R_1}{2}+2R_2+R_3+R_4\right)\\
&\le R\frac{M(a\sigma)^\sigma}{(k+\ell)^{\sigma+\gamma-1}}
\end{align*}
where
$\,
R=\frac{\mu}{c_1}+\frac{M}{c_1}\left(\frac{R_1}{2}+2R_2+R_3+R_4\right)\!.
$
We get our result as soon as we can find an $a$ satisfying
$$
R(a\sigma)^\sigma\le a^{\sigma+\gamma-1}(\sigma+\gamma-1)^{\sigma+\gamma-1}.
$$
Any $a$ with $a\ge \frac{R^{\frac{1}{\gamma-1}}}{\gamma-1}$ will do the job.
\end{proof}

\appendix
\section*{Appendix}
\renewcommand{\thesection}{A}

\subsection*{Proof of Lemma \ref{lemma:bound}}  It suffices to show that the series
\begin{equation}
\label{eq:series}
g(x_1+\ii y_1,\ldots,x_d+\ii y_d)=g(\bm{x}+\ii\bm{y})=\sum_{\bm{a}\in \N^d} 
\frac{D^{\bm{a}} f(\bm{x})(\ii\bm{y})^{\bm{a}}}{\bm{a}!}=\sum_{n=0}^\infty\sum_{|\bm{a}|=n} 
\frac{D^{\bm{a}} f(\bm{x})(\ii\bm{y})^{\bm{a}}}{\bm{a}!},
\end{equation}
where $\bm{a}!=a_1!\cdots a_d!$, $D^{\bm{a}}=\partial_{x_1}^{a_1}\cdots\partial_{x_d}^{a_d}$
and $|\bm{a}|=a_1+a_2+\cdots+a_d$,
whenever $\,\bm{a}=(a_1,\ldots,a_d)$, converges in $W_\beta$ and satisfies the 
Cauchy-Riemann equations,
for every pair $(x_j,y_j)$. 
We have
$$
|g(\bm{x}+\ii\bm{y})|\leq\sum_{n=0}^\infty\beta^n\sum_{|\bm{a}|=n}
 \frac{1}{\bm{a}!} |D^{\bm{a}} f(\bm{x})|.
$$
Since
$\displaystyle\,
D^{\bm{a}} f(\bm{x})= \sum_{\bm{k} \in\Z^d} \ii^{|\bm{a}|} \bm{k}^{\bm{a}}\hat{f}_{\bm{k}}\,
\ee^{\ii \bm{k}\cdot\bm{x}} ,\,$
and
$$
\sum_{|\bm{a}|=n}\frac{|\bm{k}^{\bm{a}}|}{\bm{a}!}=
\frac{1}{n!}\sum_{a_1+\cdots+a_d=n}\frac{n!}{\alpha_1!\cdots\alpha_d!}|k_1|^{\alpha_1}\cdots |k_d|^{\alpha_d}=
\frac{1}{n!}\big(|k_1|+\cdots+|k_d|\big)^n=
\frac{|\bm{k}|^{n}}{n!},
$$
it follows that
$$
\sum_{|\bm{a}|=n} \frac{1}{\bm{a}!} |D^{\bm{a}} f(\bm{x})|\le 
\sum_{|\bm{a}|=n} 
\sum_{\bm{k} \in\Z^d}
\frac{|\bm{k}^{\bm{a}}|}{\bm{a}!} |\hat{f}_{\bm{k}}|
=\sum_{\bm{k} \in\Z^d}
\bigg(\sum_{|\bm{a}|=n} \frac{|\bm{k}^{\bm{a}}|}{\bm{a}!}\bigg) |\hat{f}_{\bm{k}}|
= \frac{1}{n!} \sum_{\bm{k} \in\Z^d} |\bm{k}|^n |\hat{f}_{\bm{k}}|.
$$
Setting $\,\|f\|_{n,1}=\sum_{\bm{k}\in\Z^d}|\bm{k}|^n|\hat{f}_{\bm{k}}|$,\, we conclude that, for
$\bm{x}+i\bm{y}\in W_\beta$,
\begin{align*}
|g(\bm{x}+\ii\bm{y})| &\leq \sum_{n=0}^\infty\frac{\beta^n}{n!}\sum_{\bm{k}\in\Z^d}|\bm{k}|^n|\hat{f}_{\bm{k}}|=
\sum_{n=0}^\infty\frac{\beta^n}{n!}\|f\|_{n,1}\leq
c_\lambda\sum_{n=0}^\infty\frac{\beta^n}{n!}\|f\|_{n+\lambda,\infty}\\ &\leq
c_\lambda M\sum_{n=0}^\infty\frac{\beta^n}{n!}(a(n+\lambda))^{n+\lambda}=
c_\lambda Ma^\lambda 
\sum_{n=0}^\infty\frac{(n+\lambda)!}{n!}(\beta a)^n\frac{(n+\lambda)^{n+\lambda}}{(n+\lambda)!},
\end{align*}
which converges whenever $\beta ae<1$. Above we used that, 
$\sum_{\bm{k}\in\Z^d\setminus\{0\}}\dfrac{1}{|\bm{k}|^\lambda}=
c_\lambda<\infty$, provided that $\lambda>d$, and hence for $n>0$, 
$$
\|f\|_{n,1}=\sum_{\bm{k}\in\Z^d\setminus\{0\}}|\bm{k}|^n|\hat{f}_{\bm{k}}|=\sum_{\bm{k}\in\Z^d\setminus\{0\}}\frac{1}{|\bm{k}|^\lambda}
|\bm{k}|^{n+\lambda}|\hat{f}_{\bm{k}}|\le \|f\|_{n+\lambda,\infty}\sum_{\bm{k}\in\Z^d\setminus\{0\}}\frac{1}{|\bm{k}|^\lambda}=
c_\lambda\|f\|_{n+\lambda,\infty}.
$$
Following tedious but straightforward calculations we obtain that the
series in \eqref{eq:series} defines a $C^\infty$ function in all its variables, which satisfies Cauchy-Riemann
equations, in $W_\beta$, provided $\beta a\mathrm{e}<1$. \hfill $\Box$

\medskip
\noindent The following Lemma was used in the proof of our main Theorem.
\begin{lemma}
\label{lem:2}
The following sums are uniformly bounded $($i.e., they have bounds independent of $k,\ell$$)$
\begin{subex}{2}
\item $\displaystyle
\underset{0<j+m<k+\ell}
{\sum_{j=0}^k\sum_{m=0}^\ell}\dfrac{(k+\ell)^3}{(j+m)^3(k-j+\ell-m)^3},$
\item $\displaystyle
\sum_{j=0}^k\sum_{m=1}^\infty
\dfrac{(k+\ell)^3}{(j+m)^3(k-j+\ell+m)^3},$
\item $\displaystyle
\underset{\vphantom{0<j<m}}{\sum_{j=1}^\infty\sum_{m=1}^\infty}
\dfrac{(k+\ell)^3}{(j+m)^3(k+j+\ell+m)^3},$
\item $\displaystyle
\sum_{j=1}^\infty\sum_{m=1}^\infty
\dfrac{(k+\ell)^3}{(j+\ell+m)^3(k+j+m)^3}.$
\end{subex}
\end{lemma}
\begin{proof}
We will make use of the generalized mean inequality, stating that
for all $a,b,n\in\N$,
$$
(a+b)^n\leq 2^{n-1}(a^n+b^n).
$$
We will also use that for any $n,m_0\in\N$ with $n,m_0>1$,
$$
\sum_{m=m_0}^\infty\frac{1}{m^n}\leq \frac{n}{(m_0-1)^{n-1}}
$$
which follows easily by comparing the series with an appropriate integral.
For item (1) we have
\begin{align*}
\underset{0<j+m<k+\ell}
{\sum_{j=0}^k\sum_{m=0}^\ell}&\frac{(k+\ell)^3}{(j+m)^3(k-j+\ell-m)^3}=
\underset{0<j+m<k+\ell}
{\sum_{j=0}^k\sum_{m=0}^\ell}
\left(\frac{1}{j+m}+\frac{1}{k-j+\ell-m}\right)^3\\
&\leq\underset{0<j+m<k+\ell}
{\sum_{j=0}^k\sum_{m=0}^\ell}
\frac{4}{(j+m)^3}+\frac{4}{(k-j+\ell-m)^3}=
\underset{0<j+m<k+\ell}
{\sum_{j=0}^k\sum_{m=0}^\ell}\frac{8}{(j+m)^3}\\
&\leq\underset{0<j+m<k+\ell}
{\sum_{j=0}^\infty\sum_{m=0}^\infty}\frac{8}{(j+m)^3}\leq
C+\int\int_D\frac{8}{(x^2+y^2)^{\frac{3}{2}}}dxdy,
\end{align*}
where $C=\frac{8}{1}+\frac{8}{2}+\frac{8}{1}$ and $D$ is the area
$D=\{(x,y)\in\mathbb R^2:x,y\geq 0 \text{ and } x^2+y^2\geq 1\}.$

For item (2) we have that 
\begin{align*}
\sum_{j=0}^k\sum_{m=1}^\infty&
\frac{(k+\ell)^3}{(j+m)^3(k-j+\ell+m)^3}\leq
\sum_{j=0}^k\sum_{m=1}^\infty
\frac{(k+\ell+2m)^3}{(j+m)^3(k-j+\ell+m)^3}\\
&=\sum_{j=0}^k\sum_{m=1}^\infty
\left(\frac{1}{j+m}+\frac{1}{k-j+\ell+m}\right)^3\leq
\sum_{j=0}^k\sum_{m=1}^\infty
\frac{4}{(j+m)^3}+\frac{4}{(k-j+\ell+m)^3}\\
&\leq\sum_{j=0}^k\frac{12}{j^2}+\frac{12}{(k-j+\ell)^2}\leq
\sum_{j=0}^k\frac{12}{j^2}+\frac{12}{(k-j)^2}=
\sum_{j=1}^k\frac{24}{j^2}<4\pi^2.
\end{align*}

Similarly for item (3) we have that
\begin{align*}
\sum_{j,m=1}^\infty
\frac{(k+\ell)^3}{(j+m)^3(k+j+\ell+m)^3}\leq
\sum_{j,m=1}^\infty
\frac{8}{(j+m)^3}
\leq\sum_{j=1}^\infty\frac{24}{j^2}<4\pi^2,
\end{align*}
and finally, for item (4) we have that
\begin{align*}
\sum_{j,m=1}^\infty\frac{(k+\ell)^3}{(j+\ell+m)^3(k+j+m)^3}\leq
\sum_{j,m=1}^\infty\frac{8}{(j+m+1)^3}
\leq\sum_{j=1}^\infty\frac{24}{(j+1)^2}<4\pi^2.
\end{align*}

\end{proof}

\bibliographystyle{amsplain}

\end{document}